\documentclass[ leqno,12pt, twoside]{amsart}
\usepackage{geometry}
\usepackage{amssymb}

\newtheorem{theorem}{Theorem}

\newtheorem{definition}[theorem]{Definition}
\newtheorem{cor}[theorem]{Corollary}
\newtheorem{prop}[theorem]{Proposition}
\newtheorem{lemma}[theorem]{Lemma}
\theoremstyle{remark}
\newtheorem{remark}[theorem]{Remark}

\newcommand{\del}{\partial}
\newcommand{\R}[1]{{{\mathbf R}^{#1}}}
\newcommand{\Rn}{{\mathbf{R}^n} }
\newcommand{\ev}{\mathbf{E}}
\newcommand{\ind}{{\bf 1}}

\newcommand{\ift}{{\mathcal F}^{-1}}
\newcommand{\PM}[1]{P\!M^{#1}}
\newcommand{\esssup}{\mathop{\rm ess\,sup}}
\newcommand{\Fh}{{\mathcal F_h}}
\newcommand{\FhT}{{\mathcal F_{h,T}}}
\newcommand{\bv}{\mathbf{v}}

\begin{document}

\title[Exponent bounds for a convolution inequality]{Exponent bounds for a convolution inequality in Euclidean space with applications to the Navier-Stokes equations}


%    author one information
%    \author[short version for running head]{name for top of paper}
\author[C.~Orum]{Chris  Orum}
\address{University of Utah,
Department of Mathematics,
155 S 1400 E RM 233,
Salt Lake City, UT 84112-0090}
\email{orum@math.utah.edu}
\thanks{Preprint.}%This work was partially supported by the US National Science Foundation through the  Focussed Research Group collaborative award DMS-0073958 and collaborative award DMS-0940249.}

%    author two information
\author[M.~Ossiander]{Mina Ossiander}
\address{Department of Mathematics,
Oregon State University,
Corvallis, OR 97331-4605}
\email{ossiand@math.oregonstate.edu}
% \thanks{}

% \subjclass is required.
\subjclass[2010]{Primary 35Q30,  76D05,  42B37; Secondary 76M35, 39B72, 60J80}

\date{}

%\dedicatory{}

%    "Communicated by" -- provide editor's name; required.
\commby{Walter Craig}

%    Abstract is required.
\begin{abstract}
The  convolution inequality $h*h(\xi) \leq B |\xi|^\theta h(\xi)$ defined on $\Rn$ arises from 
a probabilistic representation of solutions of the $n$-dimensional Navier-Stokes equations, 
$n \geq 2$.  Using a chaining argument, we establish the nonexistence of strictly positive 
fully supported  solutions of this inequality if $\theta \geq n/2$, in all dimensions $n \geq 1$.   
We use this result to describe  a  chain of continuous embeddings from spaces 
associated with probabilistic solutions to the spaces $BMO^{-1}$ and 
$BMO_T^{-1}$ associated with the Koch-Tataru solutions of the Navier-Stokes equations.  
\end{abstract}   
\maketitle

%%%%%%%%%%%%%%%%%%%%%%%%%%%%%%%%

%\begin{keyword}  
%\PACS
%\end{keyword}
%\end{frontmatter}
\section{Introduction}
Convolution inequalities of the form 
\begin{equation}\label{eq:conv.inequality}
 h*h(\xi) = \int_\Rn h(\xi - \eta) h(\eta) d \eta \leq B |\xi|^\theta h(\xi), 
 \quad \xi \in \Rn, \quad \theta  \geq 0, \quad B>0
\end{equation} 
arise in the analysis of the incompressible Navier-Stokes equations via 
probabilistic representations of solutions.    
Our main theorem shows that if $h : \Rn \rightarrow (0, \infty]$ is a fully supported function satisfying \eqref{eq:conv.inequality}, 
 then the range of the exponent $\theta$ is constrained by the dimension $n$.  
 Letting  $\mathcal H^\theta(\Rn)$ denote the class of solutions of \eqref{eq:conv.inequality} on 
 $\Rn$ we obtain:
\begin{theorem}\label{th:dim.constraint}
If $h \in \mathcal H^\theta(\Rn)$, $n \geq 1$, then $\theta < n/2$. 
\end{theorem}

Our study of this inequality is motivated by an effort to better understand  the structure and limitations 
of the stochastic cascade representation of solutions to the Navier-Stokes equations as first 
introduced by Le Jan and Sznitman \cite{LeJan:Sznitman:97}  and then extended by later authors.  
Essentially, any $h$ satisfying \eqref{eq:conv.inequality}  induces a Banach space $\Fh$ (the initial 
value space) and another space $\FhT = B(0,T; \Fh)$ of bounded $\Fh$-valued functions defined on 
$[0,T]$ (the path space) that supports a Picard iteration scheme for establishing existence and 
uniqueness of solutions of the Navier-Stokes initial value problem.  The function  $h$, which we refer 
to as  a majorizing kernel, must be fully supported to correspond to real-valued solutions. 
Additionally, if the particular $h$ inducing $\Fh$ has $\theta = 1$, then the  solutions so obtained are 
global in time under the restriction that the data are sufficiently small.  If $h$ has $0 
\leq \theta < 1$, then the Picard iteration scheme accommodates arbitrarily large data, but the 
solutions are restricted to be local in time; i.e.\ take values in $\FhT$ where $T$ depends on the size 
of the initial datum in the $\Fh$-norm. In the case $\theta  = 1$ there is a equivalent stochastic 
cascade model providing solutions to the Cauchy problem operative for all $t \geq 0$.

The authors of \cite{FRG:03} give  examples of majorizing kernels and analyze some 
properties of classes of majorizing kernels.  In all cases considered however, the fully supported 
examples with exponent $\theta = 1$ are in dimensions $n \geq 3$. The results presented here 
provide further understanding of this phenomena by demonstrating that  there are no fully supported 
solutions with $\theta = 1$ in $\R2$.    Correspondingly, there is no \emph{direct} analogue of the 
global $3$-dimensional stochastic cascades model in dimension $n = 2$.    On the other hand, the 
Picard iteration method applied with $\FhT$, $0 \leq \theta < 1$ is  sufficient to show existence 
of classes of solutions that are local in time, in any dimension $n \geq 2$. In other words, 
Theorem \ref{th:dim.constraint} imposes no limitations on these local solutions (in any dimension $n 
\geq 2$); it only limits the approach for global solutions in dimension $n=2$.

The organization of this paper is as follows.  
Section \ref{sec:2} reviews the origins and importance of the convolution inequality (\ref
{eq:conv.inequality}).   Section \ref{sec:3} contains a short proof of Theorem \ref{th:dim.constraint}.   
In Section \ref{sec:4} we consider the continuous embeddings of certain $\Fh $ into the pseudomeasure 
spaces $\PM{n-\theta}$,  which Theorem \ref{th:dim.constraint} plays a role in establishing, as well as
successive embeddings into Besov spaces and the spaces $BMO^{-1} $ and $BMO^{-1}_T$ 
associated with the Koch-Tataru solutions.

\section{Background and motivation}\label{sec:2}

Consider the incompressible Navier-Stokes equations  
formulated as a Cauchy problem on all of $\Rn$ where $n \geq 2$. 
This system models the flow of an idealized incompressible viscous fluid issued from an initial velocity field $u_0 = u_0(x)$  at time $t = 0$.   Dimension  $n = 3$ is of central
importance, but the formulation is of interest in arbitrary dimension $n \geq 2$. The unknowns 
are the velocity vector $ u = u(x,t) = (u_i(x,t))_{i=1}^n$ and scalar pressure $p = p(x,t)$, where $x = (x_1, \dots, x_n)$.  
The system consists of $n+1$ 
coupled nonlinear equations  
\begin{subequations}\label{eq:NS}
\begin{align}\label{eq:NS1}
&\frac{\del u_i }{\del t}(x,t)  + \sum_{j= 1}^n u_j(x,t) \frac{\del u_i}{\del x_j}(x,t)
 =  \nu \sum_{j= 1}^n \frac{\del^2 u_i }{\del x_j^2}(x,t)  - \frac{\del p }{\del x_i}(x,t)  + g_i(x,t), \\[.6ex] \label{eq:NS2}
  &\sum_{j= 1}^n  \frac{\del u_j }{\del x_j}(x,t)  = 0,
  \end{align}
\end{subequations}
supplemented by the initial condition $\lim_{t \rightarrow 0} u(x,t) = u_0(x)$.    Here  $\nu $ denotes the 
kinematic viscosity and $g(x,t) = (g_i(x,t))_{i=1}^n$  is an external forcing term.  For simplicity we 
may assume that $\nabla \cdot u_0(x) = 0$ and  $\nabla \cdot g(x,t) =  0 $ for all $t
$.  

In 1997 Le Jan and Sznitman \cite{LeJan:Sznitman:97} introduced a representation of the solutions 
of a Fourier space  integral formulation of  (\ref{eq:NS}) in three spatial dimensions as a multiplicative 
functional defined on a continuous-time branching process.  These Fourier transformed Navier-
Stokes equations (\emph{FNS}) may be written 
\begin{align}\label{eq:FNS} 
 \nonumber
 \hat u(\xi,t ) ={}& e^{-\nu |\xi|^2 t} \hat u_0(\xi) + \int_0^t e^{-\nu |\xi|^2 (t-s) }   \hat g(\xi, s) ds  \\ 
 &+(2 \pi)^{-n/2} \int_0^t
|\xi| e^{-\nu |\xi|^2 (t-s) } \int_\Rn [ -\mathrm{i} \frac{\xi}{|\xi|} \cdot \hat u(\eta, s) ]   \hat {\mathbf P} (\xi) \hat u(\xi-\eta, s) d\eta ds,
\end{align}
where $\xi = (\xi_1, \dots , \xi_n)$ is the Fourier space variable, $\hat u(\xi,t)$ denotes the spatial 
Fourier transform of the unknown velocity field (and similarly for $\hat g(\xi, t)$ and $\hat u_0(\xi)$), 
and  
$\hat{ \mathbf P }(\xi)$ denotes the Leray-Helmholtz projection whose pointwise action in Fourier 
space is to project a vector $z \in \mathbf{C}^n$ onto the subspace orthogonal to $\xi \neq 0$:  
$$ \hat{\mathbf P}(\xi) z = z - (\mathbf e_\xi \cdot z) \mathbf e_\xi; \quad \mathbf e_\xi = \frac{\xi}{|
\xi|}. $$
A key device in the $\R3$  representation in \cite{LeJan:Sznitman:97} is the dimension specific rescaling of $\hat u$  and $\hat{g}$, 
\begin{equation*}
 \chi(\xi,t) = \frac{2}{\nu} \Big(\frac{\pi}{2}\Big)^{3/2} |\xi|^2 \hat u(\xi,t), \quad  \varphi(\xi, t) = \frac{4}{\nu^2} \Big(\frac{\pi}{2}\Big)^{3/2} \hat g(\xi,t),
 \end{equation*}
which allows the simultaneous description of a `splitting distribution' for a pair of particles $\{\Xi_1, \Xi_2\}$ replacing $\xi$ in the branching process, namely 
 \begin{equation}\label{eq:splitting} 
 \mathbf{Pr}(\Xi_i \in A) = \frac{1}{\pi^3} \int_A \frac{|\xi|}{|\xi - \eta|^2 |\eta |^2} d\eta, \quad A \subseteq \R3, \quad i = 1,2, \quad  \Xi_1 + \Xi_2 = \xi, 
\end{equation}
 and the normalization of $|\xi| \exp\{-\nu |\xi|^2 s\}$ to the  density of an exponential random variable describing the random lifetime of the particle of type $\xi$ so replaced.   

Details of this construction may be found in \cite{LeJan:Sznitman:97}, \cite{FRG:03}; extensions may be found in \cite{FRG:05}, \cite{Bloemker:Romito:Tribe:06}, \cite{FRG:03(2)}, \cite{FRG:08}, \cite{Morandin:05}, \cite{Ossiander:05}. Related analytical papers are  \cite{Bakhtin:Dinaburg:Sinai:04}, \cite{Gubinelli:06}, \cite{Sinai:05}, \cite{Sinai(chapter):05}, \cite{Sinai(JSP):05}.
Essentially the solution is represented in the form of an expected value
\begin{equation}\label{eq:ev.rep}
\hat u(\xi,t) = h(\xi) \mathbf{E} \mathsf X(\xi,t)
\end{equation}
where $h(\xi) = \pi^{-3}|\xi|^{-2}$  solves the convolution equation $h*h(\xi) = |\xi|h(\xi)$ on $\R3$ and  
${\mathsf X}(\xi,t)$ is defined by a backward recursion arising from a probabilistic interpretation of  
the rescaled formulation of \eqref{eq:FNS}.  Figure \ref{fig:branching} illustrates the branching 
process and the construction of the multiplicative functional $\mathsf X(\xi,t)$.  
The $\otimes_\xi$-operation performed at each of the binary nodes in the branching process encodes the algebraic 
structure of the bilinear term on the right hand side of \eqref{eq:FNS}: for two vectors $z, w \in 
\mathbf C^3$ we define $z \otimes_\xi w \in \mathbf C^3$ by $z \otimes_\xi w = - \mathrm i [z \cdot 
\mathbf  e_\xi ]  \hat{\mathbf P}(\xi) w.$
\setcounter{figure}{0}
\setlength{\unitlength}{.74mm}
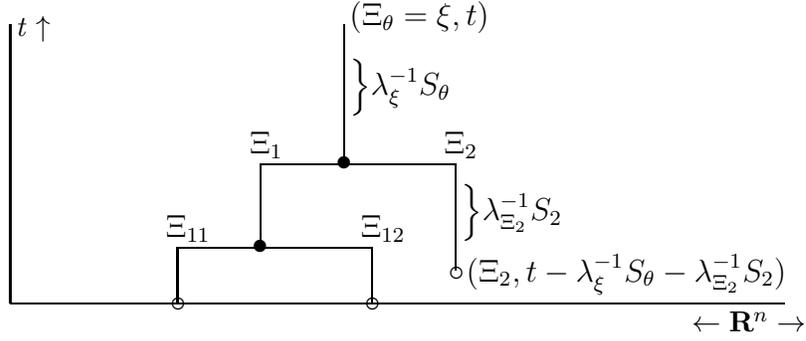
\begin{figure}[t] 
\begin{picture}(150, 70) 
\put(0,5){\line(0,1){50}}   % the 'y-axis'
\put(1,53){\small \emph{t $\uparrow$} }
\put(0,5){\line(1,0){139}}  % the 'x-axis' 
\put(122.5,0){\small $\leftarrow \Rn \rightarrow$ }
\put(60,55){\line(0,-1){25}} % the vertical line coming down from {$(\xi,t)$}
\put(61,42.5){$\Big\} \lambda_{\xi}^{-1} S_\theta$}   
\put(61,55){$(\Xi_\theta = \xi,t)$}
\put(58.5,28.9){$\bullet$}       % the top closed circle 
\put(45,30){\line(1,0){35}}  % the top horizontal line
\put(45,30){\line(0,-1){15}} 
\put(30,15){\line(1,0){35}}  %the lower horizontal line
\put(30,15){\line(0,-1){10}} 
\put(28.7,3.6){$\circ$}
\put(63.7,3.6){$\circ$}          %the middle  open circle 
\put(65,15){\line(0,-1){10}}
\put(43.5,13.9){$\bullet$}       %the lower 'closed bullet' node
\put(80,30){\line(0,-1){19}} %the vertical line on the right
\put(81,20){$\Big\} \lambda_{\Xi_2}^{-1} S_2$} 
\put(78.7,9.1){$\circ$}       %the right open circle
\put(82,9){$(\Xi_2, t - \lambda_\xi^{-1} S_\theta  -\lambda_{\Xi_2}^{-1} S_2)$}
\put(43,32){$\Xi_1$}   
\put(78,32){$\Xi_2$}
\put(28,17){$\Xi_{11}$}      
\put(63,17){$\Xi_{12}$}
\end{picture}
\caption{\label{fig:branching} A schematic illustration of the branching process and 
the construction of 
$\mathsf X(\xi,t)$:  a particle of type $\xi = \Xi_\theta$ lives for a random length of time $\lambda^{-1}_\xi S_\theta$ and then dies out.  Depending on the outcome of a Bernoulli random variable with mean $1/2$, it is either not 
replaced at all or replaced by two correlated particles $\Xi_1$ and $\Xi_2$ distributed as $\eqref{eq:splitting}$, or more generally
\eqref{eq:splitting.gen}. 
The two new particles in turn live for independent random lifetimes, and so the 
process continues. There are two types of nodes: input nodes ($\circ$) accept data when a 
particle dies out without replacement or when its lifetime extends below the horizontal axis at time $t = 0$.  
Operational nodes ($\bullet$) combine data according to  $z, w \mapsto m(\Xi_\bv) z \otimes_{\Xi_
\bv} w$ and send the output upward.  Here $m(\xi)$ is a multiplicative factor that arises from the 
rescaling  of $\hat u$ by $h$,  $\lambda_\xi = \nu |\xi|^2$, $\bv \in \{\theta, 1, 2, 11, \dots \}$, and $S_\theta, S_1, S_2, \dots$ are \emph{i.i.d.}\ standard exponential random variables. }
\end{figure}

This representation provides existence and uniqueness results for the solutions of 
\eqref{eq:FNS} in the space of pseudomeasures  $(P\!M^2)^3$.   The scale of pseudomeasure spaces is defined by
\begin{equation}
 P\!M^a = \big\{ f \in \mathcal S^\prime(\Rn): \hat f \in L^1_{\text{loc}}(\Rn), \,
 \| f ; {P\!M^a}\| = \esssup_{\xi \in \Rn} |\xi |^a |\hat f(\xi) | < \infty \big\}
\end{equation}
where $a \geq 0$ is a given parameter and  
$\mathcal S^\prime(\Rn)$ denotes the space temperate distributions  on $\Rn$.  Alternatively, the spaces $\PM{a}$ may be regarded as homogeneous Besov-type spaces based on the classical space of pseudomeasures $\PM{} = \PM{0}$:
$$P\!M^a = \dot{B}^{a,\infty}_{\PM{}} = \{ f \in \mathcal S^\prime(\Rn) : \textstyle   \sup_{j \in \mathbf{Z}} 2^{aj} \| {\Delta_j f} \|_{\PM{}} < \infty \}.$$ 
Here $\Delta_j f$ is the \emph{j}th dyadic block of the Littlewood-Paley decomposition of $f$.

Returning to the stochastic cascade, it was later recognized  \cite{Cannone:00} that the same existence and uniqueness results  could be 
obtained by applying the Picard iteration argument  with the Banach space $B(0,T; (P\!M^2)^3 )$ 
of bounded functions $f(t) :[0,T] \rightarrow (P\!M^2)^3$.    
  This argument is notable for the continuity of 
\begin{equation} \label{eq:bilinear}
\mathbf B = \mathbf B(u, v)(x,t) =\int_0^t e^{\nu (t - s) \Delta} \mathbf{P} \nabla \cdot ( u \otimes  v) (s)  ds. 
\end{equation}
That is,  
$\mathbf B : E \times E \rightarrow E $ is continuous if $E =   B(0,T; (P\!M^2)^3 )$  but it is not continuous in general, in which case the Picard iteration argument typically requires the use of an embedded subspace with a second norm, 
see e.g.\ \cite[p.~220]{Bahouri:Chemin:Danchin:11}, \cite{Cannone:04}, \cite{Lemarie-Rieusset:02}, \cite{Meyer:06}.   

The authors of \cite{FRG:03} generalize this approach 
by showing that  the $h$ in \eqref{eq:ev.rep}  may belong to a more general class of  \emph{FNS majorizing kernels} which are positive solutions  of 
\eqref{eq:conv.inequality} parameterized by the exponent $\theta$.   
There are two natural Banach spaces associated with a given majorizing kernel.  The first is the majorization  space  
\begin{equation}\nonumber
\Fh = \big\{ f \in \mathcal S^\prime(\Rn)^n : 
        \hat f(\xi) \in L^1_{\text{loc}}(\Rn)^n, 
        \| f ; \Fh \| = \sup_{\xi \in \Rn} [h(\xi)]^{-1} | \hat f(\xi)| < \infty \big\}.
\end{equation}
The initial data for the Cauchy problem belongs to this space.
The second is the path space which contains the solutions:   $\FhT = B(0, T; \Fh)$  denotes the bounded $\Fh$-valued functions defined on the interval $[0,T]$ with norm 
$$\| f(t) ; \FhT \| = \sup_{0 \leq t \leq T } \| f(t) ; \Fh \|. $$ 
It turns out that 
$\mathbf B: \FhT \times \FhT \rightarrow \FhT $ is  again continuous for any majorizing kernel of 
exponent $0 \leq \theta \leq 1$, and the  Picard iteration argument is directly applicable without the introduction of a second norm.  This 
yields global 
existence and uniqueness results in the case $\theta = 1$ (with small data) and local existence and 
uniqueness results in the case $0 \leq \theta < 1$ (with arbitrarily large data).  The argument 
works for the \emph{FNS} equations formulated in any dimensions $n \geq 2$ subject to the 
constraint  $\theta < n/2$   of Theorem \ref{th:dim.constraint}.

For majorizing kernels with exponent $\theta =1$, the Picard iteration scheme can be connected to 
stochastic cascades 
as follows.   A  solution $h$ of \eqref{eq:conv.inequality} 
provides the following splitting distribution for a branching process generalizing \eqref{eq:splitting}:
\begin{equation}\label{eq:splitting.gen}
\mathbf {P r}(\Xi_1 \in A)  = \int_A \frac{h(\xi - \eta ) h(\eta)} {h*h(\xi)} d\eta,  \quad A \subseteq \Rn, \quad \Xi_1 + \Xi_2 = \xi.      
\end{equation} 
This specializes to \eqref{eq:splitting} 
in dimension $n = 3$ with  $h(\xi) = \pi^{-3}|\xi|^{-2}$.  For the more general branching processes and $\mathsf X(\xi,t)$ defined accordingly, one can define a sequence of events $\{G_k\}_{k\geq 0}$ pertaining to the branching 
process so that the sequence $\ev ( \mathsf X(\xi,t); G_k) $ and the iterates of the Picard contraction 
argument are in one-to-one correspondence.  This is established in \cite{FRG:03} with the conclusion 
that for global solutions with $\theta = 1$, the existence of the expected value representation and the 
convergence of the Picard iteration 
scheme are essentially equivalent.    

\section{Main theorem}\label{sec:3} 
The majorizing kernels considered in \cite{FRG:03} are allowed to be supported on  various convex additive semigroups $W \subset \Rn$.   Here however we focus entirely 
on fully supported majorizing kernels:  those $h$ 
 with  $  \int_A h(\xi) d\xi > 0$ 
 for all subsets $A \subseteq \Rn$ having positive Lebesgue measure.  Fully supported majorizing kernels correspond to majorization spaces $\Fh$ and $\FhT$ that contain real data and  solutions.

\begin{definition} \label{def:maj.ker} A majorizing kernel with exponent $\theta$ is a tempered function (a tempered distribution that is also a function)  $h : \Rn \rightarrow (0, \infty]$ satisfying the following conditions:
\begin{enumerate}
\item[(i)] $h * h (\xi) \leq B |\xi|^\theta h(\xi)$ for all $\xi \in \Rn$ with constants $B > 0$ and $\theta \geq 0$;  
\item[(ii)]  $h^{-1}(\infty)$ has $n$-dimensional Lebesgue measure zero. 
\end{enumerate}
 We set $\Omega_h = \Rn \setminus h^{-1}( \infty )$.
  If $B$ is such that
 $\sup_\xi {h*h(\xi)}/{|\xi|^\theta h(\xi)} = B$,
 then $B$ is  {sharp}.  If $B=1$ is sharp  then  $h$ is {standardized}.  
\end{definition}
The set of majorizing kernels of exponent $\theta$ defined on $\Rn$ is denoted $\mathcal H^{\theta}(\Rn)$. 
It is possible for a given majorizing kernel to have a range of exponents. 

The invertible map $h \mapsto B^{-1}h$ on $\mathcal H^\theta(\Rn)$ has the effect of standardizing 
any non-standardized majorizing kernel with sharp constant $B$; hence for the purpose of proving 
Theorem \ref{th:dim.constraint}  we assume, without loss of generality,  that $h$ is standardized.  We 
also make this assumption in the proofs of the  lemmas in this section.  

\begin{lemma}
Let $h \in \mathcal H^\theta(\Rn)$, $\theta \geq 0$. Then for all $R >0$,  $h(\xi)$ is bounded away from zero on the closed ball $\overline{B}(0, R) = \{ \xi \in \Rn : |\xi| < R\}$.  
\end{lemma}
\begin{proof} We assume that $h$ is standardized.  Fix $R > 0$, and define  $g(\xi) : \Rn \rightarrow [0, 1]$ by restricting and truncating $h(\xi)$:    
\begin{equation}\nonumber
g(\xi) = 
\begin{cases}
0  & \text{ if } \xi \notin B(0, R), \\
\min \{ h(\xi), 1 \} & \text{ if } \xi \in B(0,R).    
\end{cases}
\end{equation} 
Since $h(\xi) \geq g(\xi) \geq 0 $ for all $\xi \in \Rn$,  we have  $h*h(\xi) \geq g*g(\xi) \geq 0$ as well.  Let $\xi_0 \in \overline{B}(0, R)$, the closed ball.  
Since $0< g(\xi)  \leq 1$ on $B(\xi_0/2, R/2) \subset B(0, R)$,  it follows that 
\begin{equation}\nonumber
g*g(\xi_0) \geq  \int_{|\eta | < R/2} g\big(\frac{\xi_0}{2} - \eta\big) g\big(\frac{\xi_0}{2} + \eta\big) d\eta > 0.
\end{equation}    In other words, for  $J = g*g (\overline B(0,R))$ we have $J \subset (0,\infty)$.  But since $g \in L^2(\Rn)$ it follows that $g*g(\xi)$ is continuous on $\Rn$, hence $J$ is compact and connected, i.e., $J$ is  a closed subinterval of $ (0, \infty)$ that is necessarily bounded away from zero.   Then  on $\overline B(0, R)$ we have 
 $$h(\xi) \geq  |\xi|^{-\theta} h * h(\xi) \geq  |\xi|^{-\theta} g * g(\xi) \geq   R^{-\theta} g*g(\xi), $$  
giving that $h(\xi)$ is also  bounded away from zero on $\overline B(0,R)$.  
\end{proof}
\begin{cor}\label{cor:1}
If $h \in \mathcal H^\theta(\Rn)$, $\theta \geq 0$ then for all $\xi_0 \in \Omega_h$, $\xi_0 \neq 0$ there exist  $\delta = \delta(\xi_0) > 0$ and $\varepsilon = \varepsilon(\xi_0) > 0$     such that $\inf_{| \eta | < \delta} h(\xi_0 - \eta) \geq \varepsilon h(\xi_0)$.
\end{cor} 
In the following two lemmas $B^*(r)= \{ \xi \in \Rn : 0 < |\xi| < r\}$ denotes the punctured ball of radius $r>0$ in $\Rn$ centered at the origin. 
\begin{lemma} 
\label{lem:0.or.infty}
Suppose $h(\xi) \in \mathcal H^\theta(\Rn)$ with  $ \theta \geq n/2$.  Then $h(\xi)$ has the following behavior at the origin:  either  
$\liminf_{\xi \rightarrow 0 } |\xi|^\theta h(\xi) = 0$ or 
$ \liminf_{\xi \rightarrow 0 } |\xi|^ \theta h(\xi)  = \infty$.
\end{lemma}
\begin{proof}
Assume $h$ is standardized.  Suppose that $\liminf_{\xi \rightarrow 0} |\xi|^\theta h(\xi) > 0$.  Then for some $L > 0$ there exists a $\delta > 0$ such that $|\xi|^\theta h(\xi) > L$ for all $\xi \in B^*(\delta)$.   Then for $\xi \in B^*(\delta)$,
\begin{align}
\nonumber
 |\xi|^\theta h(\xi) \geq   
 \int\limits_{|\eta | < \delta/2} 
 h(\eta) h(\xi - \eta ) d\eta   \geq  
  \int\limits_{|\eta | < \delta/2}
 \frac{L^2 d\eta }
        {|\eta |^\theta |\xi - \eta  |^\theta}.
\end{align}
Applying Fatou's Lemma with any convergent sequence $\xi_n \rightarrow 0$  gives
\begin{align} \nonumber \label{eq:liminf}
 \liminf_{\xi_n \rightarrow 0} |\xi_n|^\theta h(\xi_n) \geq  
       \int\limits_{|\eta_1| < \delta/2}    \liminf_{\xi_n \rightarrow 0}\frac{L^2 d\eta } {|\eta |^\theta |\xi_n - \eta  |^\theta} =   
 \int\limits_{|\eta| < \delta/2} 
\frac{L^2 d\eta }{|\eta |^{2\theta} }    
= \infty.
\end{align}  
\end{proof}
\begin{lemma}
\label{lem:not.infty}
If $h  \in \mathcal H^\theta(\Rn)$ and  $\theta \geq n/2 $ then $ \liminf_{\xi \rightarrow 0 } |\xi|^ \theta h(\xi)  \neq \infty$.  
\end{lemma}
\begin{proof}    Fix $n \geq 1$ and take $\theta \geq n/2$.  Suppose for contradiction  that there exists $h  \in \mathcal H^\theta(\Rn)$ with $\liminf_{\xi \rightarrow 0} |\xi|^\theta h(\xi) = \infty$.  
We may assume $h$ is standardized.   
For $x>0$, let 
\[ \rho(x) = \sup \{ r>0 : |\xi|^\theta h (\xi) > x \;\;  \forall   \xi \in B^*( r)\}.\]
Notice that $\rho$ is a non-increasing function of $x$ with $\lim_{x \to \infty} \rho (x) = 0$.  Furthermore, for any $\xi \in B^*( \rho(x)/2)$, 
\begin{align}
 \label{calc} 
 \nonumber
 |\xi|^\theta h(\xi) \geq   \int_{|\eta|<\rho(x)/2} h(\eta) h(\xi-\eta) d \eta 
 \geq {}&   \int_{|\eta|<\rho(x)/2}\frac{ x^2 }{|\eta|^{\theta}  |\xi-\eta|^{\theta}} d \eta  \\
  &\geq    \frac{x^2} {[\rho(x)]^{\theta} }\int_{|\eta|<\rho(x)/2} \frac{d\eta}{|\eta|^{\theta}}.
\end{align}

The cases $\theta \geq n$ and $n/2 \leq \theta < n$, are now considered separately.  We can easily dispense with the first case.  If $\theta \geq n$ and $\rho(x) >0$, then
$  \int_{|\eta|<\rho(x)/2} |\eta|^{-\theta} d \eta = \infty$,
giving $h(\xi) \equiv \infty$ on $B^*(\rho(x)/2)$  and violating Definition \ref{def:maj.ker}.  Therefore if $h \in \mathcal H^\theta(\Rn)$ and $\theta \geq n$ then $\liminf_{\xi \rightarrow 0} h(\xi) \neq \infty. $

For $\theta \in [n/2, n)$, the calculation in \eqref{calc} gives
\begin{equation}\label{calc2}
|\xi|^\theta h(\xi)  \geq   C_{n,\theta}  x^2 [\rho(x) ]^{n-2\theta} \quad \text{for}\quad  \xi \in B^*(\rho(x)/2) 
\end{equation}
 where $C_{n,\theta} = \pi^{n/2} ( (n-\theta) 2^{n-\theta} \Gamma(n/2) )^{-1}$ depends only on $n$ and $\theta$.  
This will give a lower bound on the rate at which $|\xi|^\theta h(\xi) \to \infty$ as $|\xi| \to 0$.  In particular if we  define  
\begin{equation}
\nonumber
 \lambda (x) =  C_{n, \theta}  [\rho(x)]^{n-2\theta},
\end{equation}
then inspection of \eqref{calc2} gives
\begin{equation} \label{eq:rhobound}
\rho(x^2 \lambda(x) ) \geq 2^{-1} \rho(x).  
\end{equation}

We now define a rapidly increasing sequence $\{x_k \}$ iteratively in a way that allows inequality  \eqref{eq:rhobound} to control the corresponding decrease in $\rho(x_k)$.   This will yield a contradiction with the assumption that $\liminf_{\xi \rightarrow 0} |\xi|^\theta h(\xi) = \infty$.  
First observe that for $n/2 < \theta < n$, $\lambda(x)$ is a non-decreasing function of $x$ with 
$\lambda(x)  \rightarrow \infty $ as $x \rightarrow \infty$.  If $\theta = n/2$ then 
$\lambda (x) = C_{n, n/2} = 2 (\pi/2)^{n/2}  \left( n \Gamma(n/2) \right)^{-1}$.
Accordingly these two cases are treated separately in defining $\{x_k \}$.    
For $\theta \in (n/2 , n)$, fix $x_0 \geq 2$ large enough to also have $\lambda(x_0) \geq 2$ and define  $\{x_k\}_{k \geq 1}$ iteratively via
\[x_k = x^2_{k-1} \lambda( x_{k-1} ) \quad \text{for} \quad k \geq 1.\]
Then $x_1 \geq 2 x_0^2 \geq 2^3$ and by induction, 
\[x_k \geq 2 x^2_{k-1} \geq 2^{2^{k+1} -1}. \]
For $\theta = n/2$, take $x_0 = \max\{2, 2 (C_{n,{n/2}} )^{-1} \}$ and define  $\{x_k\}_{k \geq 1}$ iteratively via
\[x_k = C_{n,{n/2}}   x^2_{k-1}  = x^2_{k-1}  \lambda(x_{k-1}) \text{ for } k \geq 1.\]
Then $x_1 = \max\{ 2^2 C_{n,{n/2}}, 2^2 (C_{n,{n/2}})^{-1} \}$ and by induction, 
$$ x_k = \max\{ 2^{2^k} (C_{n,n/2})^{2^k-1},  2^{2^k} (C_{n,{n/2}})^{-1} \}.$$ 
Since $\max\{ (C_{n,n/2})^{2^k-1},  (C_{n,{n/2}})^{-1} \} \geq 1$, in both cases the sequence satisfies both
\begin{equation} \label{eq:sequence.a}
x_k \geq 2^{2^k},  \quad  k \geq 0
\end{equation}
and 
\begin{equation} \label{eq:sequence.b}
 x_k = x^2_{k-1} \lambda( x_{k-1} ), \quad  k \geq 1.
 \end{equation}
In particular \eqref{eq:sequence.a}  gives $\rho(x_k) \to 0$ as $ k \to \infty$.  On the other hand \eqref{eq:sequence.b} combined with \eqref{eq:rhobound} controls the rate of decrease of $\rho(x_k)$ via
\[ \rho(x_k) \geq 2^{-1} \rho(x_{k-1}) \geq 2^{-k} \rho(x_0).\]

To conclude the proof, fix $\xi_0 \in \Omega_h$, $\xi_0 \neq 0$.  By Corollary \ref{cor:1} there exists $\delta, \epsilon > 0$ such that 
$\inf_{|\eta | < \delta} h(\xi_0-\eta) \geq \epsilon h(\xi_0).$
Take $k$ large enough to have $\rho(x_k) < \delta$.  Then
\begin{align}\label{eq:xi.0.h} 
 \nonumber 
 |\xi_0|^\theta h(\xi_0) \geq{}&
  \int\limits_{|\eta| < \rho(x_k) }  
 h(\eta) h(\xi_0 - \eta)   d\eta \; \geq \; \epsilon h(\xi_0) \int\limits_{|\eta| < \rho(x_k)  }    
 h(\eta)  d\eta
 \\ 
 & \geq   \epsilon h(\xi_0)  \int\limits_{|\eta| <\rho(x_k) } 
	\frac{x_k}{|\eta |^\theta } d\eta  \; = \;  \epsilon  h(\xi_0)   C^\prime_{n,\theta} x_k  \rho^{n-\theta} (x_k) \\ \nonumber 
	& \geq \epsilon  h(\xi_0)   C^\prime_{n,\theta} 2^{2^{k} } \left( 2^{-k} \rho(x_0) \right)^{n-\theta}
\end{align}	
where $C^\prime_{n,\theta}  = 2^{n-\theta}C_{n,\theta}   $  depends only on $n$ and $\theta$.
As the right-hand side of  \eqref{eq:xi.0.h} becomes arbitrarily large as $k \rightarrow \infty$, contradicting the finiteness of the left-hand side, we conclude that if $h \in \mathcal H^\theta(\Rn)$ with $\theta \geq n/2$ then $\liminf_{\xi \rightarrow 0} |\xi|^\theta h(\xi) \neq \infty$.  
\end{proof}

\begin{proof}[Proof of Theorem \ref{th:dim.constraint}] 
Suppose for contradiction 
$h \in \mathcal H^\theta(\Rn)$  with  $\theta \geq n/2$.  Assume $h$ is standardized.   By Lemmas \ref{lem:0.or.infty} and \ref{lem:not.infty}, $\liminf_{\xi \rightarrow 0} |\xi|^\theta h(\xi) = 0$. 
 Applying Fatou's Lemma we find 
\begin{align} \nonumber 
 0 =  \liminf_{\xi \rightarrow 0} |\xi|^\theta h(\xi)  
 \geq{}& 
 \liminf_{\xi \rightarrow 0}  \int_\Rn h(\eta ) h(\xi- \eta) d \eta \\ 
 \nonumber
 & \geq 
 \int_\Rn h(\eta ) \liminf_{\xi \rightarrow 0}  h(\xi-  \eta) d \eta  
 = \int_\Rn h(\eta ) \liminf_{\xi \rightarrow - \eta }  h(\xi) d \eta ,  
\end{align}
implying that for almost all $\eta \in  \Rn$  
\begin{equation}\label{eq:liminf2}
  h(\eta) \liminf_{\xi \rightarrow - \eta }  h(\xi) = 0.
 \end{equation} 
In particular this holds for almost all $\eta \in \Rn$ such that $-\eta \in \Omega_h.$   
For any such  $-\eta  \in \Omega_h$, $h(\xi)$ is bounded away from zero in a neighborhood of $- \eta $ by Corollary \ref{cor:1} and since  $h(\eta) > 0$ for all $\eta $, 
\begin{equation}
\nonumber
 h(\eta) \liminf_{\xi \rightarrow - \eta }  h(\xi) >  0.
 \end{equation} 
This contradicts \eqref{eq:liminf2}.  Therefore if $\theta \geq n/2$ then $\mathcal H^\theta(\Rn) = \emptyset$, or equivalently, if $h \in \mathcal H^\theta(\Rn)$ then $\theta < n/2$.  
\end{proof}

\section{Embedding properties and relation to Koch-Tataru solutions}
\label{sec:4}
In this section we discuss  properties of majorizing kernels and majorization spaces implied by Theorem \ref{th:dim.constraint}. One consequence in particular, is that for a given  majorizing kernel that behaves algebraically at the origin and at infinity we have the continuous embedding $\Fh \hookrightarrow \PM{n-\theta}$  (a slight modification is needed if $\theta = 0$).   This is part of a chain of continuous embeddings from $\Fh$ up to the spaces $BMO^{-1}$, $BMO^{-1}_T$ and $\overline{VMO}{}^{-1}$, the initial value spaces for the Koch-Tataru solutions of the Navier-Stokes equations.   If $\theta = 1$ then we have
\begin{equation}
\nonumber
 \Fh \hookrightarrow \PM{n-1} \hookrightarrow \dot{B}^{-1 + \frac{n}{p},  \infty}_p \hookrightarrow BMO^{-1}  \quad(n < p < \infty),
\end{equation} 
and if $0 < \theta < 1$ then for  $0< T < 1$ we have
\begin{equation}  
\nonumber
\Fh \hookrightarrow \PM{n-\theta} \hookrightarrow \dot{B}^{-\theta + \frac{n}{p},  \infty}_p \hookrightarrow B^{-\theta + \frac{n}{p},  \infty}_p  \hookrightarrow BMO_T^{-1}  \quad  (\frac{n}{\theta} < p < \infty).   
\end{equation} 
Here, as in the sequel, it is convenient to ignore the distinction between spaces of scalar-valued and 
vector-valued functions.  We follow the index convention in \cite{Grafakos.vol.2:08}  for the homogeneous and inhomogeneous Besov spaces, $\dot B_p^{s, q}$ and $B_p^{s, q}$ respectively. 
     Although the successive embeddings after $\Fh 
\hookrightarrow \PM{n-\theta}$ are known, we record them here for completeness and to emphasize 
how the dichotomy between cases $\theta < 1$ and $\theta = 1$ aligns with the dichotomy between local 
and global solutions in the endpoint spaces.  Finally, we note that there are majorization spaces 
$\Fh \not\subset \PM{n-\theta}$ that embed further along these chains.  Proposition \ref{prop:besov}  
illustrates how this can occur.  
Since the class of majorizing kernels has itself  not  been completely characterized we are not able to 
locate all majorization spaces on scales of classical or better known Banach spaces.

\subsection{The endpoint spaces}
Recall that Koch and Tataru \cite{Koch:Tataru:01} consider the iteration scheme for solving the mild 
formulation of the Navier-Stokes equations in the largest critical space of tempered distributions  
subject to the condition that  $e^{t \Delta} u_0$ belong to $ L^2_{\text{loc}}(\Rn \times [0, \infty))$ so that the bilinear term makes sense.    
This is the function space $BMO^{-1} = BMO^{-1}(\Rn)$ admitting a Carleson measure characterization through the norm 
 \begin{equation}
 \nonumber
\| f \|_{BMO^{-1}} = \sup_{x, \, R >0} \bigg( \frac{1}{|B(x,R )| }  \int_0^{R^2} \int_{B(x,R)} |e^{t\Delta} f |^2 dy dt \bigg)^{1/2}. 
 \end{equation} 
Here $B(x, R)$ denotes the ball of radius $R$ centered at $x \in \Rn$, and $| B(x,R) |$ is its 
Lebesgue measure.  Equivalently, $BMO^{-1}$ consists of functions that can be written as the  
divergence of vector fields whose components belong to $BMO$, the  space of functions of bounded 
mean oscillation.
 
In \cite{Koch:Tataru:01} it is shown that given sufficiently small  initial datum $u_0 \in BMO^{-1}$,  there exists a mild solution of the Navier-Stokes equations issued from $u_0$ in the path space  
$X$ of functions defined on $\Rn \times \R{}^+$ with norm
\begin{equation}
\nonumber
\| u \|_X = \sup_t t^{1/2} \| u(t) \|_{L^\infty(\Rn)} +\sup_{x, \, R > 0} \bigg(\frac{1}{|B(x,R )| } \int_0^{R^2} \int_{B(x, R)} |u|^2 dy dt \bigg)^{1/2}.   
\end{equation}
It is also shown that there exists a constant $\varepsilon_0$ such that for all $\|u_0 \|_{BMO_T^{-1}} < \varepsilon_0$ there exists a mild solution in the local path space $X_T$, defined by the norm 
\begin{equation}
\nonumber
\| u \|_{X_T} = \sup_{0 < t < T}  t^{1/2} \| u(t) \|_{L^\infty(\Rn)} +  \sup_{x,\; 0 < R^2 < T} \bigg( \frac{1}{|B(x,R)|} \int_0^{R^2} \int_{B(x, R)} |u|^2 dy dt \bigg)^{\textstyle \frac{1}{2}}.   
\end{equation}
Here $BMO_T^{-1}$ is defined as $BMO^{-1}$ except that we only consider balls of size $\sqrt{T}$ and smaller: 
\begin{equation}
\nonumber
\| f \|_{BMO_T^{-1}} = \sup_{x, \, 0 < R^2 < T} \bigg( \frac{1}{|B(x, R)| }  \int_0^{R^2} \int_{B(x,R)} |e^{t\Delta} f |^2 dy dt \bigg)^{1/2}.   
\end{equation}
Finally,   ${\overline{VMO}}{}^{-1} : = \big\{ f \in BMO^{-1}_1 : \| f \|_{BMO^{-1}_T} \rightarrow 0 \text{ as } T \rightarrow 0 \big\}$, following the notation of \cite{Koch:Tataru:01}. In \cite{Miura:05} and elsewhere, this space is denoted by $vmo^{-1}$.         
\subsection{Continuous embeddings of  majorization spaces}  We first define a class of majorizing 
kernels having certain algebraic growth and decay properties, and then consider  continuous 
embeddings  of the  associated majorization spaces.   
\begin{definition}
A majorizing kernel $h \in \mathcal H^\theta(\Rn)$  \emph{behaves (algebraically) at the origin as $|\xi|^{-\alpha}$} or  \emph{blows up (algebraically) at the origin as $|\xi|^{-\alpha}$} if there exists an $\alpha  \geq 0$ such that $h(\xi) = O(|\xi|^{-\alpha})$ as $\xi \rightarrow 0$ and $|\xi|^{-\alpha} = 
O(h(\xi))$ as $\xi \rightarrow 0$.       
\end{definition}

\begin{definition}\label{def:omega}
A majorizing kernel $h \in \mathcal H^\theta(\Rn)$ \emph{behaves (algebraically) at infinity as $|\xi|^{-\omega}$} or  \emph{decays (algebraically) at infinity as $|\xi|^{-\omega}$} if there exists an $\omega > 0$ such that $h(\xi) = O(|\xi|^{-\omega})$ as $\xi \rightarrow \infty$ and $|\xi|^{-\omega} = O(h(\xi))$ as $\xi \rightarrow \infty$.      
\end{definition}

\begin{definition} 
\label{def:r.alg}A majorizing kernel has \emph{radial algebraic growth and decay} if 
\begin{enumerate}
\item it behaves algebraically at the origin as $|\xi|^{-\alpha}$ for some $\alpha \geq 0$; 
\item it decays algebraically at infinity as $|\xi|^{-\omega}$  for some $\omega > 0$;
\item it is bounded on the complement of any neighborhood of the origin.
\end{enumerate}   The subclass of $\mathcal H^\theta(\Rn)$ consisting of those majorizing kernels with radial algebraic growth and decay is denoted $\mathcal H_{\alpha,\omega}^\theta (\Rn)$.  
\end{definition}

Note that if $ h \in \mathcal H_{\alpha,\omega}^\theta(\Rn)$ then for any  $\lambda > 0$ the elements $f \in \Fh$ satisfy the scaling relation 
\begin{equation}\label{eq:scaling2}
 C_\lambda\| f \|_\Fh \leq \| \lambda^\theta f_\lambda  \|_{\Fh} \leq C^\prime_\lambda  \| f \|_{\Fh}
 \end{equation}  where $f_\lambda(x)  = f(\lambda x)$ and constants $C_\lambda$ and $C_\lambda^
 \prime$ may depend on $\lambda$.  On the other hand this property does not characterize  $
 \mathcal H_{\alpha,\omega}^\theta (\Rn)$.  The majorizing kernels discussed in 
Proposition \ref{prop:products} for example,  satisfy \eqref{eq:scaling2} with $C_\lambda \equiv C_
\lambda^\prime \equiv 1$, yet do not belong to $\mathcal H_{\alpha,\omega}^\theta (\Rn)$.  Majorization spaces also exist that satisfy \eqref{eq:scaling2} only for certain $\lambda$; taking 
$h(\xi) = (2 \pi)^{-1} |\xi|^{-1} e^{-|\xi|} \in \mathcal H^1(\R3)$ for example, 
yields a majorization space $\Fh$ that satisfies \eqref{eq:scaling2} only for $\lambda \leq 1$.  In this case, 
$ \lambda \| f \|_\Fh \leq \| \lambda  f_\lambda \|_{\Fh} \leq \lambda^{-1}  \| f \|_{\Fh}$. 

\begin{theorem}\label{th:alpha} Suppose $h \in \mathcal H^\theta(\Rn)$ blows up at the origin as $|\xi|^{-\alpha}$.  If $0 <  \theta  < n/2$  then  $\alpha \leq n - \theta$, and  if $\theta = 0$ then $\alpha < n$. 
\end{theorem}
\begin{proof}  
Assume $h \in \mathcal H^\theta(\Rn)$ behaves as $|\xi|^{-\alpha}$ at the origin and is standardized.  
Then $\alpha < n$ lest $h*h(\xi) \equiv \infty$.    Suppose for contradiction $\alpha > n-\theta$. Then  
$2 \alpha > n$ by Theorem \ref{th:dim.constraint}.    Since $|\xi|^{-\alpha} = O(h(\xi))$ as $\xi 
\rightarrow 0$, there exist constants $R > 0$ and $ C > 0 $ such that $|\xi |^{-\alpha} \leq C h(\xi)
$ for all $|\xi| < R$. Then for all $|\xi| < R/3$ we have 
\begin{align}\label{eq:origin.behavior} \nonumber
C \geq  |\xi|^{\alpha - \theta}h*h(\xi) 
 \geq{}& C_1 |\xi|^{\alpha - \theta} \int_{|\eta| < 2 R/3} \frac{d \eta}{|\xi - \eta |^{\alpha} |\eta|^{\alpha}} \\ 
& \geq C_1 |\xi|^{\alpha - \theta} \bigg\{ \int_\Rn \frac{ d \eta}{|\xi-\eta|^{\alpha}|\eta|^{\alpha }} - \int_{|\eta| > 2 R/3}  \frac{ d \eta}{|\xi-\eta|^{\alpha}|\eta|^{\alpha }} \bigg\} \\ \nonumber
& \geq C_1\bigg\{ \frac{C_2}{|\xi|^{\alpha - (n -\theta)} }- C_3 |\xi|^{\alpha - \theta}\bigg\},   
\end{align}
using the convolution  equality $|\xi|^{-\alpha} * |\xi|^{-\alpha} = C_2 |\xi|^{-(2 \alpha -n)}$ and the  estimate
\begin{equation}\nonumber
\int_{|\eta| > 2 R/3} \frac{ d \eta}{|\xi-\eta|^{\alpha}|\eta|^{\alpha }} \leq C_3 < \infty,
\end{equation} 
both of which hold if $ n/2 <  \alpha < n$.   
Then the right hand side of \eqref{eq:origin.behavior} tends to $+\infty$ as $\xi \rightarrow 0$ which contradicts 
the finiteness of the left hand side.  Hence if $h(\xi)$  behaves as $|\xi|^{-\alpha}$ at the origin, then $\alpha \leq n-\theta$.     
\end{proof}
\begin{theorem}\label{th:omega}
If $h(\xi) \in \mathcal H^\theta(\Rn)$ decays  at infinity as  $|\xi|^{-\omega}$ then $\omega \geq n- \theta$.  
\end{theorem}

\begin{proof}
Assume $h$ behaves as $|\xi|^{-\omega}$ at infinity and is  standardized.  There exists an $R >0$ and constants $C_1$, $C_2$  such that $C_1 |\xi|^{-\omega} \leq h(\xi) \leq C_2 |\xi|^{-\omega}$ whenever $ |\xi| > R$. 
Using the fact that $2 |\xi - \eta| \geq |\eta|$ whenever  $|\eta | \geq 2 |\xi|$, we have for all $\xi$ such that $|\xi| > R$, 
\begin{align}
\nonumber
C_2 |\xi|^{\theta-\omega} \geq  | \xi|^\theta h(\xi) \geq h*h(\xi) \geq{}& \int_{|\eta| \geq 2 |\xi|}h(\xi - \eta) h(\eta) d\eta \\
\nonumber
& \geq 2^{-\omega} C_1^2 \int_{|\eta| \geq 2 |\xi|} |\eta|^{-2 \omega}  d\eta = C |\xi|^{n-2\omega}. 
\end{align}
This implies $2 \omega > n$ and   $|\xi|^{n-2\omega} = O(|\xi|^{\theta - \omega})$ as $\xi \rightarrow \infty$,  hence $\omega \geq \max \{ n/2,  n  - \theta \}.$  The maximum here is superfluous by  Theorem \ref{th:dim.constraint}, and we have  $\omega \geq n - \theta > n/2$.  
\end{proof}

Given Banach spaces $X$ and $Y$ we write $X \hookrightarrow Y$ 
to denote the continuous embedding of $X$ into $Y$.   

\begin{theorem} 
\label{th:emb1}
Suppose $h \in \mathcal H^\theta_{\alpha, \omega}(\Rn)$, $\theta < n/2$ and $n \geq 2$.    
If $\theta = 1$ then $ \Fh \hookrightarrow \PM{n-1} \hookrightarrow BMO^{-1}$.
If  $0  <  \theta < 1$, then  $\Fh \hookrightarrow \PM{n-\theta} \hookrightarrow BMO_T^{-1}$ for all $T > 0$.  
\end{theorem}
\begin{remark} 
The following intermediate  embeddings are also known:   for $p \in (n,\infty)$, $n \geq 2$,  
$$
\PM{n-1} \hookrightarrow \dot B_p^{-1 + \frac{n}{p}, \infty} \hookrightarrow BMO^{-1}.
$$
See e.g.\ \cite[p.\ 267]{Cannone:Karch:04}, \cite[p.~228]{Bahouri:Chemin:Danchin:11}, respectively.  For $p \in (n/\theta, \infty)$, $n \geq 2$,  $0 < \theta < 1$ and  $0 < T \leq 1$, we also have
\begin{equation*}  
 \PM{n-\theta} \hookrightarrow \dot{B}^{-\theta + \frac{n}{p},  \infty}_p \hookrightarrow B^{-\theta + \frac{n}{p},  \infty}_p  \hookrightarrow BMO_T^{-1}, 
\end{equation*}
through a suitable  modification of \cite[Lemma 7.1]{Cannone:Karch:04} and \cite[Remark 4.2f.]
{Koch:Tataru:01}.  
\hspace{\fill}{\tiny $\blacksquare$}  
\end{remark}  

Recall that for the pair  $\{L^q = L^q(\Rn), L^r = L^r(\Rn) \}$,  $1 \leq q , r \leq \infty$,  the set  $L^q + L^r = \{ f_q + f_r : f_q \in L^q, f_r \in L^r \}$  becomes a Banach space when equipped with the norm 
$$ \| f \|_{L^q + L^r} = \inf \{ \| f_q \|_{L^q} + \| f_r \|_{L^r} : f = f_q + f_r \}.$$

\begin{proof}[Proof of Theorem \ref{th:emb1}]
We first consider simultaneously    
$\PM{n-1} \hookrightarrow BMO^{-1}$ ($n \geq 2$)  and  $\PM{n-\theta} \hookrightarrow BMO^{-1}_T$,  ($n \geq 2$, $0 < \theta < 1$).  Recall  $f_\lambda(x) = f(\lambda x)$. By scaling we have 
\begin{equation}\label{eq:scaling}
\frac{1}{|B(x_0, \lambda)| }  \int_0^{\lambda^2} \int_{B(x_0,\lambda)} |e^{t\Delta} f |^2 dy dt   =\frac{\lambda^{n+2-2\theta} }{|B(x_0, \lambda)| }\int_0^1 \int_{B(\lambda^{-1}x_0, 1)} |  e^{t\Delta}  \lambda^\theta f_\lambda  |^2 dy dt.
\end{equation}
The following estimate holds for all $f \in \PM{n-\theta}$, $ (n \geq 2, 0 < \theta \leq 1)$:  
\begin{equation}\nonumber
\Big( \int_{B(x_0,1)}  |e^{t \Delta } f |^2 dy \Big)^{\textstyle \frac{1}{2}} \leq     \|f^{(2)} \|_{L^2} + |B(x_0, 1) |^{\textstyle \frac{1}{2}} \| f^{(\infty)} \|_{L^\infty}   \leq C \| f \|_{\PM{n-\theta}} \| h^\star  \|_{L^1 + L^2}. 
\end{equation}
Here $h^\star = |\xi|^{-(n-\theta)} \in L^1 + L^2$ and $f = f^{(2)}+ f^{(\infty)} \in L^2 + L^\infty$.
Applying this estimate to the right-hand side of \eqref{eq:scaling} along with the scaling relation $\| \lambda^\theta f_\lambda \|_{\PM{n-\theta}} = \| f \|_{\PM{n-\theta}}$  gives 
\begin{align}
\nonumber
\| f \|_{BMO_T^{-1}} &\leq C_\theta T^{(1-\theta)/2 } \| f \|_{\PM{n-\theta}},   \quad (0 < \theta  <  1), \\
\nonumber
\| f \|_{BMO^{-1}} &\leq C_\theta \| f \|_{\PM{n-1}}, \quad (\theta = 1),
\end{align}
where $C_\theta = C \| h^\star  \|_{L^1 + L^2}$ depends on $\theta$.   
Now suppose $h \in \mathcal H^\theta_{\alpha,\omega}(\Rn)$, $\theta < n/2$, $n \geq 2$ and $0 < \theta \leq 1$.   
Theorems \ref{th:alpha}  and \ref{th:omega} imply   $\sup_{\xi \in \Rn}|\xi|^{n-\theta}h(\xi) < \infty$ and  
then $\Fh \hookrightarrow \PM{n-\theta}$ follows.      
\end{proof}
\begin{remark} \label{rem:0} Theorems  \ref{th:alpha}  and \ref{th:omega} also imply that if $h \in \mathcal H^{0}_{\alpha, \omega}(\Rn)$, then there exists a constant $\sigma$  (depending on $h$) with  $0 < \sigma < 1$, such that $\Fh \hookrightarrow \PM{n-\sigma}$.   \hspace{\fill} {\tiny $\blacksquare$}
\end{remark} 

The next theorem deals with embeddings $\Fh \hookrightarrow \overline{VMO}{}^{-1}$, where the latter space plays a role in the local solutions analyzed in \cite{Koch:Tataru:01}.      We set $\log_+ x = \max\{ \log x,  0\}$.  We write $\ind[\cdot]$ for the indictor function and $\ift$ for the inverse Fourier transform.  

\begin{theorem}
\label{th:emb2}
Suppose $h \in \mathcal H^\theta_{\alpha, \omega}(\Rn)$, $n\geq 2$.  If $0  <  \theta < 1$ then there exists a constant $C_\theta$ (which may depend on $\theta$) such that for all $f \in \Fh$, $ T > 0$, 
\begin{equation}
\label{eq:est}
\| f \|_{BMO^{-1}_T} \leq C_\theta T^{(1-\theta)/2} \| f \|_{\Fh},
\end{equation}hence $\Fh \hookrightarrow  {\overline{VMO}}{}^{-1}$.
If $\theta = 0$ then there exist a constant $C$  such that 
\begin{equation}\label{eq:theta.zero.est}
\| f \|_{BMO^{-1}_T} \leq 
\begin{cases}
C T^{1/2} \| f \|_{\Fh} & \text{ if } \omega > n,  \\[.6ex]
C T^{1/2} ( 1 + \log_+ T^{-1/2})^{1/2} & \text{ if } \omega = n
\end{cases} 
\end{equation} for all $f \in \Fh$, $ T > 0$, hence $\Fh \hookrightarrow  {\overline{VMO}}{}^{-1}$. 
\end{theorem}

\begin{proof}  
The proof of Theorem \ref{th:emb1} implies \eqref{eq:est} for $0 < \theta < 1$.
If $\theta = 0$ and  $\omega > n$ then $h \in L^1$.   Then for all $f \in \Fh$, $\lambda > 0$,  we have
$$
\Big( \int_{B(\lambda^{-1} x_0,1)}  | e^{t \Delta } f_\lambda |^2 dy \Big)^{\textstyle \frac{1}{2}} 
\leq    |B(\lambda^{-1} x_0, 1) |^{\textstyle \frac{1}{2}} \| f_\lambda  \|_{L^\infty} 
\leq C \| f \|_{L^\infty}  \leq C \| f \|_{\Fh} \| h  \|_{L^1}. 
$$
Applying this to \eqref{eq:scaling} gives the first part of \eqref{eq:theta.zero.est}. 
If $\theta = 0$ and  $\omega = n$ we let  
$f_\lambda = f_\lambda^{ (2) } + f_{\lambda}^{(\infty)} \in L^2 + L^\infty$ where 
$$
f_\lambda^{(2)} = f_\lambda^{(2, R)}  
 = \ift (\, \widehat{f_\lambda}\,  \ind{[\, |\xi| \geq R \, ]  }\, ), 
 \quad f_\lambda^{(\infty)} =f_\lambda^{(\infty, R)} 
  = \ift (\,\widehat{f_\lambda}\,  \ind{ [ \, |\xi| < R \, ] } \, ),
$$ 
and  $R $ is chosen  so that $C_1 |\xi|^{-n} \ind_{[ \, |\xi | \geq R \, ]} \leq h(\xi) \ind_{[ \, |\xi | \geq R \, ]} \leq C_2 |\xi|^{-n} \ind_{[ \,  |\xi | \geq R \, ]}$. Then if $\lambda < 1$ we have 
\begin{align}\nonumber 
\Big( \int_{B(\lambda^{-1} x_0,1)} \! | e^{t \Delta } f_\lambda^{(2)} |^2 dy \Big)^{\textstyle \frac{1}{2}} \leq{}& 
C \sup_{|\xi| > R} \frac{\lambda^{-n} \hat f (\lambda^{-1} \xi)}{\lambda^{-n} h(\lambda^{-1} \xi)}
\sup_{| \xi | > R} \frac{\lambda^{-n} h (\lambda^{-1} \xi)}{ h( \xi)} \| h \ind_{[ \, |\xi| \geq  R \, ]} \|_{L^2} \\
\nonumber
& \leq C \frac{C_2}{C_1}  \| f \|_{\Fh} \| h \ind_{[\, |\xi| >  R\, ]}  \|_{L^2},  \quad (\lambda < 1),  
\end{align}
and
\begin{align} \nonumber 
\Big( \int_{B(\lambda^{-1} x_0,1)}  | e^{t \Delta } f_\lambda^{(\infty)} |^2 dy \Big)^{\textstyle \frac{1}{2}}   \leq{}&  | B(\lambda^{-1} x_0 , 1 ) |^{\textstyle \frac{1}{2}} \| f_{\lambda}^{(\infty)}   \|_{L^\infty} \\ 
\nonumber 
&\leq 
C \| f \|_{\Fh} \int_\Rn h(\xi) \ind{[\, \lambda |\xi| <  R\, ]} d\xi \\
\nonumber
& \leq C( 1 + \log \lambda^{-1} ) \|f \|_\Fh  \quad (\lambda < 1).
\end{align}
Applying these estimates to \eqref{eq:scaling} gives the second part of \eqref{eq:theta.zero.est} provided $T \leq 1$,  which is sufficient to conclude  $\Fh \hookrightarrow \overline{VMO}{}^{-1}$.    
In addition, for $T > 1$, we have 
$$ \| f \|_{BMO^{-1}_T} \leq C T^{(1-\sigma)/2} \|   f   \|_{\PM{n-\sigma}} \leq C T^{(1-\sigma)/2} \|   f   \|_{\Fh}  \leq C T^{1/2} \|   f   \|_{\Fh}, \quad (T > 1),$$
on account of the  embedding $\Fh \hookrightarrow \PM{n-\sigma}$ as noted in Remark \ref{rem:0}.
 \end{proof} 
 
\subsection{Further properties} 
Theorems \ref{th:emb1}  and \ref{th:emb2} use the fact that  if $h \in \mathcal H_{\alpha,\omega}^
\theta(\Rn)$ then $h \in L^1 + L^2$.   Not all majorizing kernels share this property.  Proposition \ref
{prop:products} provides a class of counterexamples, making use of the following criterion for $L^q + 
L^r$:  a measurable function $f$ defined on $\Rn$ belongs to $L^q + L^r$, $1 \leq q < r \leq 
\infty$ if and only if for all $M > 0$, $f \boldsymbol 1_{\boldsymbol [ |f| \geq M ]} \in L^q$ and  $f 
\ind_{[|f| \leq M]} \in L^r$.
  \begin{prop}  \label{prop:products}
 Let $n \geq 2$, $k \geq 2$ and $\vartheta <  n/2$ and partition the coordinate index set of $\Rn$ into $k$ blocks:  $\{ 1, \dots , n\} = I_1 \cup \dots \cup I_k$, $|I_i| = d_i$, $\sum d_i = n$.   Define  $ h (\xi)$ on $\Rn$  by 
\begin{equation}
\nonumber
h(\xi) = \prod_{i = 1}^k r_i^{-(d_i - \theta_i)},  \quad r_i^2 = \sum_{j \in I_i}  \xi_j^2,  \quad  \xi = (\xi_1 \dots \xi_n), 
\end{equation} 
where $\sum_{i = 1}^k \theta_i =\vartheta$, $0 < \theta_i < d_i/2$.    Then $h \in \mathcal H^\vartheta(\Rn)$ and  $h \notin L^1 + L^2$.
 \end{prop}
 \begin{proof}
For $h$ as defined above and any set $A \subset \R+$, we have
\begin{equation} \nonumber
\int_\Rn h(\xi) \ind{[h \in A]} d\xi 
  = C \int\limits_0^\infty \cdots \int\limits_0^\infty 
  \Big( \prod_{i = 1}^k r_i^{\theta_i-1}\Big)  
  \ind{\big[\textstyle \prod_{i = 1}^k r_i^{-(d_i - \theta_i)} \in A  \big]} dr_1 \cdots dr_k 
\end{equation}
where $C$ is a constant depending only on $(d_1,\dots ,d_k)$. In particular,
 \begin{equation}
 \int_\Rn h(\xi) \ind{[h \geq 1]} d\xi =  C \int\limits_0^\infty \cdots \int\limits_0^\infty \Big( \prod_{i = 1}^k r_i^{\theta_i-1} \Big) \ind{\big[ \textstyle \prod_{i = 1}^k r_i^{d_i - \theta_i} \leq 1 \big]}  dr_1 \cdots dr_k = \infty
\end{equation}
   so $h \notin L^1 + L^2$.  On the other hand $h$, defined on $\Rn \simeq \R{d_1} \times \cdots \times  \R{d_k}$, has the form $h = \prod_i^k h_i$ with each $h_i \in \mathcal H^{\theta_i}(\R{d_i})$ so 
 \begin{equation}\nonumber
 h * h (\xi ) \leq B_1 \cdots B_k | r_1 |^{\theta_1} \ldots  | r_k |^{\theta_k} h(\xi) \leq B_1 \cdots B_k |\xi|^\vartheta h(\xi) 
 \end{equation} and  $h \in \mathcal H^\vartheta(\Rn)$.  
\end{proof} 
The next proposition shows how $ \Fh \hookrightarrow B^{-\theta +\frac{n}{p}, \infty}_p$, $\Fh \not\subset \PM{n-\theta}$  can occur.  
Following \cite[Lemma 7.1]{Cannone:Karch:04}, we use the heat semigroup characterization of the 
spaces $\dot B^{-\alpha, \infty}_p$, $\alpha > 0$ via the norm $\| f \|_{\dot B^{-\alpha, \infty}_p} = 
\sup_{t \geq 0} t^{\alpha/2} \| e^{t \Delta  } f \|_{L^p}$,  (see \cite[p.\ 72]{Bahouri:Chemin:Danchin:11}).
\begin{prop}
\label{prop:besov}
Let $n \geq 2$, $\vartheta < n/2$, and suppose $h \in \mathcal H^\vartheta(\Rn)$ is  defined as in Proposition \ref{prop:products}.  Then for all $p > n / \vartheta$, we have $\Fh \hookrightarrow \dot B^{-\vartheta + (n/p), \infty}_p$, but for $a \geq 0$, $\Fh \not\subset \PM{a}$.     
\end{prop}
\begin{proof} First $\Fh \not\subset \PM a$, $a \geq 0$ since $\Omega_h \neq \{0\}$.    
Let $n / \vartheta < p < \infty$ and set $q^{-1} + p^{-1} = 1$ so that $1 < q < n/(n-\vartheta)$.  By the Hausdorff-Young inequality, for all $f \in \Fh$, $t > 0$ we have 
\begin{align} \nonumber
\| e^{t \Delta} f \|_{L^p}^q \leq{}& C \int_\Rn | e^{-t |\xi|^2} \hat f(\xi) |^q d\xi \\
& \leq C \Big( \sup_{\xi \in \Rn } \frac{ |\hat f(\xi) |}{h (\xi)} \Big)^q \prod_{i=1}^k \Big( \int_0^\infty r^{d_i - (d_i - \theta_i)q} e^{-q r^2 t} \frac{dr}{r} \Big) \\ \nonumber
& \leq C \| f \|_{\Fh}^q \prod_{i = 1}^k t^{-(d_i - (d_i -\theta_i) q) / 2} = C \| f \|_{\Fh}^q t^{(-n + (n-\vartheta)q)/2}.  
  \end{align}
  Then $\sup_{t > 0} t^{(\vartheta - n/p)/2} \| e^{t \Delta} f \|_{L^p} \leq C \| f \|_\Fh$. 
\end{proof}
\section{acknowledgement}
This work grew out of the collaborative research effort of the authors of \cite{FRG:03} funded by the U.S.\ National Science Foundation.    We are grateful to the anonymous referees for carefully
reading a previous version of the paper  and making suggestions that resulted in significant improvements.

\end{document}